\documentclass{svjour3}                     
\smartqed  
\usepackage{amsmath,amsfonts,amssymb,graphicx}
\usepackage{tikz}
\usepackage{multirow}
\usepackage[normalem]{ulem}

\newcommand*{\B}[1]{\mathbf{#1}}

\newcommand*{\bu}{\mathbf{u}}
\newcommand*{\bv}{\mathbf{v}}
\newcommand*{\bp}{\mathbf{p}}
\newcommand*{\bq}{\mathbf{q}}
\newcommand*{\bff}{\mathbf{f}}
\newcommand*{\bX}{\mathbf{X}}
\newcommand*{\bV}{\mathbf{V}}
\newcommand*{\bU}{\mathbf{U}}
\newcommand*{\bRT}{\mathbf{RT}}
\newcommand*{\Div}{\mbox{div}\:}
\newcommand*{\Divv}{\mbox{div}~}

\newcommand*{\secref}[1]{\S\ref{#1}}

\usepackage{latexsym}
\journalname{}
\begin{document}

\title{Explicit {\em a posteriori} and {\em a priori} error estimation for the finite element solution of Stokes equations
\thanks{The first author is supported by Japan Society for the Promotion of Science, Grant-in-Aid for Scientific Research (B) 16H03950, 20H01820 and Grant-in-Aid for Scientific Research (C) 18K03411.
The second author is supported by Grant-in-Aid for Scientific Research (C) 18K03434. The last author is supported by JST CREST Grant Number JPMJCR14D4, Japan.}
}

\titlerunning{Explicit error estimation for the finite element solution of Stokes equations}        

\author{Xuefeng LIU \and Mitsuhiro NAKAO \and Chun'guang YOU \and Shin'ichi OISHI}


\institute{Xuefeng LIU \at
            Graduate School of Science and Technology, 
            Niigata University, Niigata 950-2181, Japan\\
            \email{xfliu@math.sc.niigata-u.ac.jp}           
           \and
           Mitsuhiro NAKAO\at
           Faculty of Science and Engineering, Waseda University, Tokyo 169-8555, Japan\\
              \email{m.nakao4@kurenai.waseda.jp} 
           \and
           Chun'guang YOU\at
           He attended this research when he was a Ph.D. student at Academy of Mathematics and Systems Science, Chinese Academy of Sciences, China. \\
           \email{youchg@lsec.cc.ac.cn}
           \and 
           Shin'ichi OISHI \at
           Faculty of Science and Engineering, Waseda University, Tokyo 169-8555, Japan\\
           \email{oishi@waseda.jp}
}

\date{June 4, 2020}

\maketitle

\begin{abstract}
  For the Stokes equation over 2D and 3D domains,  explicit {\em a posteriori } and {\em a priori} error estimation are novelly developed for 
  the finite element solution. 
  The difficulty in handling the divergence-free condition of the Stokes equation is solved by 
  utilizing the extended hypercircle method along with the Scott-Vogelius finite element scheme. 
  Since all terms in the error estimation have explicit values, by further applying the interval arithmetic and verified computing algorithms, 
  the computed results provide rigorous estimation for the approximation error.
  As an application of the proposed error estimation, the eigenvalue problem of the Stokes operator is considered and 
  rigorous bounds for the eigenvalues are obtained.
  The efficiency of proposed error estimation is demonstrated by solving the Stokes equation on both convex and non-convex 3D domains.
  
\end{abstract}

\section{Introduction}
  
The error estimation theory for approximate solutions to the Stokes equation is one of fundamental problems in numerical analysis for fluid simulation. 
For example, in the approach of investigating the solution to the Navier--Stokes equation by verified computing, an explicit error estimation for the approximate solution to the Stokes equation is desired. 
In \cite{Watanabe_etal1999}, by providing the {\em a priori} error estimation for the Stokes equation, Watanabe--Yamamoto--Nakao developed an algorithm to verify the solution existence for a stationary Navier--Stokes equation over a 2D square domain. 
However, for general 2D domains and further 3D domains, the {\em a priori} error estimation for Stokes equation is not yet available, which remains to be the bottleneck problem for the solution verification of the Navier--Stokes equation.

The main difficulty in the {\em a priori} error estimation is due to the divergence-free condition required in the Stokes equation. 
The classical study on the numerical solutions to Stokes equation usually involves bounded but unknown constants in the error estimation terms; see, for example, the pioneer work in \cite{verfurth1989posteriori}.
For 2D domains, the Korn inequality (see, e.g., \cite{Nakao_etal1998}) has been utilized to construct the {\em a priori } error estimation for 
star-shaped 2D domain. However, for a 2D domain with general shape, and the domains in 3D space, e.g.  a cube, it is still an open problem to give explicit values for the constant in Korn's inequality.

This paper is an approach to solve the bottleneck problem in the solution verification of the Navier--Stokes equation. 
We apply the Scott-Vogelius type finite element method (FEM) \cite{scott1985norm,zhang2005new,Zhang-MC-2014} to obtain a divergence-free approximation to the Stokes equation and
then propose an explicit {\em a priori} error estimation for the Stokes equation over general 2D and 3D domains.
In our proposed error estimation, the idea of the hypercircle method has been utilized to take the advantage of the divergence-free property of the approximate solution and further construct the explicit error estimation.
The hypercircle method, also named by the Prage-Synge theorem, has been used in the error estimation for the Poisson equation (see \cite{Kikuchi+Saito2007,Liu+Oishi2013,li2018explicit}); 
the error estimation here can be regarded as a direct extension of the result of Liu--Oishi \cite{Liu+Oishi2013}.\medskip

The features of proposed method can be summarized as follows:
\begin{itemize}
\item By combining the extended hypercircle method and the Scott-Vogelius FEM scheme, one can obtain explicit error estimation for the finite element solution to the Stokes equation. 
Since all terms in the error estimation have explicit values, by further applying the verified computation technique, the computed rigorous results can be further applied in the solution verification for the Navier--Stokes equation.

\item The proposed method only utilizes the $H^1$ information of the weak solutions, which enables the application of the error estimation to Stokes equations over general 2D and 3D non-convex domains, in which cases, the solution usually contains singularity.

\end{itemize}

As an application of the explicit error estimation proposed in this paper, we 
consider the eigenvalue problem of the Stokes equation and provide explicit lower and upper bounds for the eigenvalues of the Stokes operator. 
See the detailed discussion in \secref{sec:application-to-eigenvalue} and numerical results in \secref{sub-sec:application-to-eigenvalue-pro}.

The construction of the rest of the paper is as follows. 
Lemma 1 in \secref{sec:ns-eq} introduces the extended hypercircle over a proper function space setting.
\secref{sec:fem-space} describes the FEM spaces to be used in the solution approximation and error estimation. 
In \secref{sec:error-estimation}, the {\em a posteriori} and {\em a priori} error estimation are proposed.
In \secref{sec:application-to-eigenvalue}, an application of the {\em a priori} error estimation to the eigenvalue problem of the Stokes operator is provided.
\secref{sec:numerical-results} displays the computation results for 3D domains.

\section{Function spaces and problem setting}
\label{sec:ns-eq}
%
To make the argument concise, we only consider the equation over a 3D domain, while the 2D case can be regarded as a special case and can be processed in an analogous way. 
Generally, a bounded Lipschitz domain  will be preferred in numerical analysis. 
However, to have the domain completely partitioned by tetrahedra, the domain is assumed to a polyhedron in solving practical problems.

Let $L^p(\Omega)$ ($p>0$) and $H^k(\Omega)$, $H^k_0(\Omega)$ ($k=1,2,\cdots$) be the standard Sobolev spaces over $\Omega$. The inner product in $L^2(\Omega)$ is denoted by 
$(\cdot, \cdot)_\Omega$ or $(\cdot, \cdot)$; the $L^2$ norm of a function in $L^2(\Omega)$ is denoted by $\|\cdot\|$. The space $L_0^2(\Omega)(\subset L^2(\Omega))$ has function with degenerated average over the domain, that is, $L_0^2(\Omega):=\{v\in L^2(\Omega)~|~ (v,1)_\Omega =0\}$. 

Let us introduce the divergence-free space $\bV$ by 
\begin{equation}
\label{def:V}
\bV=\{ \bv \in ( H_0^1(\Omega))^3 ~|~ \mbox{ div } \bv = 0\}\:,
\end{equation}
where $\mbox{\Div}\bv$ denotes the divergence of vector $\bv$. 
The inner product and the norm of $\bV$ are defined by
$$
(\bu,\bv)_\bV:=\int_\Omega \nabla \bu \cdot \nabla \bv ~ \mbox{d}\Omega, ~~ \|\bu\|_{\bV} :=\sqrt{(\bu,\bu)_\bV} \:.
$$
Here, $\nabla \bu$ denotes the gradient of function $\bu$.
The $H(\Div\!\!)$ space is defined by 
$$
H(\Div;\Omega) :=\{ p \in L^2(\Omega) ~|~ \Div p \in L^2(\Omega)\}\:.
$$
We further introduce the space $(H(\Div;\Omega))^3$, the member function $\bp =(p_1, p_2, p_3)$ of which has the divergence as $\Div \bp :=(\Div p_1, \Div p_2, \Div p_3) \in (L^2(\Omega))^3$.

\medskip

In this paper, we consider the Stokes equation in a weak formulation: 
Given $\bff \in (L^2(\Omega))^3$, let \(\mathbf{u} \in \mathbf{V}\) be the exact solution such that
\begin{equation}
    \label{eq:stokes}
    (\nabla \bu, \nabla \bv) = (\bff, \bv), \quad \forall \bv \in \bV.
\end{equation}
The solution existence and uniqueness of the above equation can be easily confirmed by applying the Lax-Milgram theorem. 
The saddle point formulation utilizing test function space $L_0^2(\Omega)$ is given by (see, e.g., \cite[\S5.1]{girault2012finite} and \cite{boffi2013mixed}): 

\medskip

Find $\bu \in (H_0^1(\Omega))^3$ and $\rho \in L_0^2(\Omega)$ such that
\begin{equation}
  \label{eq:stokes-saddle-problem}
  (\nabla \bu, \nabla \bv) + (\Divv \bv, \rho) + (\Divv \bu, \eta )= (\bff, \bv), \quad \forall \bv \in (H_0^1(\Omega))^3, \eta \in L_0^2(\Omega)\:.
\end{equation}
Such a formulation will be used in the FEM approximation in \secref{sec:fem-space}. 
As the objective of this paper,  we will consider an conforming FEM approximation to (\ref{eq:stokes-saddle-problem}) and provide explicit  error estimation.

\medskip

Below, let us introduce an extended version of the hypercircle method, which will help to construct an explicit error estimation for the Stokes equation.
 
\begin{lemma}[Extended Prager-Synge's theorem]
  \label{lemma:prager-synge}
	Given $\bff \in (L^2(\Omega))^3$, let $\bu$ be the solution to (\ref{eq:stokes}) corresponding to $\bff$.
	Suppose that \(\bp \in H(\Div;\Omega)^{3}\) satisfies,
\begin{equation}
  \mbox{\em div } \mathbf{p} + \nabla \phi + \mathbf{f} =0, \text{ for certain } \phi \in H^1(\Omega)\:.  
\end{equation}
%
Then for any \(\mathbf{v} \in \mathbf{V}\), the following Pythagoras equation holds,

\begin{equation}
\label{eq:hypercircle-eq}
\| \nabla \mathbf{u} - \nabla \mathbf{v} \|^2 + \|\nabla \mathbf{u} - \mathbf{p} \|^2 = \|\mathbf{p} - \nabla \mathbf{v} \|^2 \:. 
\end{equation}
\end{lemma}
\begin{proof}
The holding of the equality can be confirmed by the expansion of 
$\|\bp - \nabla \bv\|^2=\|(\nabla \bu - \bp ) -  \nabla (\bu -\bv )\|^2$, which has 
the cross term as zero due to the divergence-free condition and the boundary condition of $\bu$ and $\bv$: 
$$
(\nabla \bu - \bp, \nabla (\bu -\bv ))
= (\bff + \Divv \bp , (\bu -\bv )) 
= (-\nabla \phi, (\bu -\bv ) ) = (\phi, \Divv (\bu - \bv))=0\:.
$$
\end{proof}
\begin{remark}
The selection of $\bp$ and $\bv$ in (\ref{eq:hypercircle-eq}) is not unique.
It is easy to see that, for a fixed solution $\bu$, to minimize $\| \bp - \nabla \bv\|$ is equivalent to minimize both $\|\nabla \bu -\nabla \bv\|$ and $\|\nabla \bu - \bp\|$ independently.
\end{remark}

\section{Finite element spaces}
\label{sec:fem-space}

In this section, we introduce the FEM spaces to be used in the solution approximation and error estimation. 

Let $\mathcal{T}^h$ be a regular tetrahedron subdivision for domain $\Omega$. Further requirement to the mesh for the purpose of a stable computation of the Stokes equation will be explained afterward.
On each element $K \in \mathcal{T}^h$, denote by $P^m(K)$ the set of polynomials with degree up to $m$.
We choose the Scott-Vogelius type finite element method to construct 
divergence-free FEM spaces to approximate the space $\bV$, 

\paragraph{Discontinuous space $\bX_h$ of degree $d$}
Let $X_h^{(d)}$ be the set of piecewise polynomials of degree up to $d$. 
Let $\mathbf{X}^{(d)}_h:=(X_h ^{(d)})^3$. Let $X_{h,0}^{(d)}:=L_0^2(\Omega)\cap X_h^{(d)}$.

\paragraph{Conforming FEM space $\bU_h(\subset \left(H^1(\Omega)\right)^3)$ and $\bV_h (\subset \bV)$ of degree $k$.}

\begin{itemize}
\item Let $U^{(k)}_h$ be the space consisted of piecewise polynomials of degree up to $k$, which also belong to $H^1(\Omega)$. That is, $U^{(k)}_h:=H^1(\Omega)\cap X^{(k)}(\Omega)$. Define $\mathbf{U} ^{(k)}_h:=({U}^{(k)}_h)^3$.
\item Let $U^{(k)}_{h,0}:=\{ u_h \in U^{(k)}_h ~|~ u_h = 0 \mbox{ on } \partial \Omega \}$, $\mathbf{U}^{(k)}_{h,0} := (U^{(k)}_{h,0} )^3 $.
\item Let $\mathbf{V}^{(k)}_{h}$ be the subspace of $\mathbf{U} ^{(k)}_{h,0}$ with divergence-free member function. That is,
$\mathbf{V}^{(k)}_{h} = \{\bu_h\in \mathbf{U}^{(k)}_{h,0} ~|~ \mbox{div }\bu_h=0 \}=\mathbf{U} ^{(k)}_{h} \cap {\bV}$.
\end{itemize}

\paragraph{The Raviart-Thomas FEM space $\bRT_h$ of degree $m$}
Define $RT^{(m)}_h$ by  
$$
RT^{(m)}_h:=\{ p_h \in H(\mbox{div};\Omega) \:\:|\:\: p_h|_K =
\mathbf{a}_K+b_K\mathbf{x}, \forall K \in \mathcal{T}^h \}\:.
$$
Here, $\mathbf{a}_K\in (P^m(K))^3, b_K\in P^m(K)$.
Define the tensor space $\bRT^{(m)}_h := ( RT_h^{(m)} )^3$.

\medskip

For the  FEM spaces defined here, the following properties hold.
\begin{equation}
\label{eq:fem_space_properties}  
\mbox{div }(\bRT_h^{(m)}) = \bX_h^{(m)}, \quad \nabla (U_h^{(k)}) \subset \bX_h^{(k-1)}\:.
\end{equation}
We may omit the superscript of degree in the notation for FEM spaces to have, for example, $\bX_h$, $\bV_h$, $\bRT_h$.

\paragraph{Construction of $\bV_{h}$ }

Generally, it is difficult to construct $\mathbf{V}_{h}$ directly due to the divergence-free condition. We turn to utilize the 
Scott-Vogelius type FEM space \cite{scott1985norm}, which handles the divergence-free condition implicitly by utilizing the test functions.
$$
\mathbf{V}^{(k)}_{h} = \{ \mathbf{v} \in \mathbf{U}^{(k)}_{h,0} \:| \: (\Divv \mathbf{v}, \eta _h)=0 \:\: \forall \eta_h \in X_h^{k-1} \}\:.
$$
The approximation to the Stokes equation with $\mathbf{V}_{h}$ reads: 
Find $\bu_h \in \mathbf{V}_{h}$ such that
\begin{equation}
  (\nabla \bu_h, \nabla \bv_h) = (\bff, \bv_h)\quad \forall \bv_h \in \mathbf{V}_{h}\:.
\end{equation}
The saddle point formulation is given by: Find $\bu_h \in \bU^{(k)}_{h,0}$, 
$\eta_h  \in X^{(k-1)}_{h,0}$, s.t., 
\begin{equation}
\label{eq:saddle-point-pro-fem}  
(\nabla \bu_h, \nabla \bv_h) + (\Divv \bv_h, \eta_h) + 
(\Divv \bu_h, \rho_h) = (\bff, \bv_h)~~\forall \bv_h \in \bU_{h,0}^{(k)}, \rho_h \in X_{h,0}^{(k-1)} \:.
\end{equation}

To have the inf-sup condition hold  for the above saddle point problem, 
we apply the method of S. Zhang \cite{zhang2005new} to create the tetrahedra division of domains (see detailed description in \S \ref{sec:numerical-results}) and select the degree of FEM spaces as below; 
\begin{equation}
  d =  m = k-1, k \ge 3\:.
\end{equation}

Let $P_h$ be the projection $P_h:\bV \to \bV_h$ such that, for any $\bv \in \bV$
\begin{equation}
    \label{eq:def-Ph}
    (\nabla (\bv - P_h \bv) , \nabla \bv_h) = 0, \quad \forall \bv_h \in 
\bV_h.
\end{equation}
Thus, the solution $\bu_h$ of (\ref{eq:saddle-point-pro-fem}) is just $\bu_h = P_h \bu$.


\section{Explicit error estimation for FEM solutions to the Stokes equation}
\label{sec:error-estimation}

In this section, we consider the 
{\em a posteriori} and the {\em a priori }
error estimation for the finite element method solution to the Stokes equation. 

As a preparation, let us introduce the constant $C_{0,h}$, which is used in the error estimation of the $L^2$-projection $\pi_h :L^2(\Omega)^3 \to \bX_h$: for any $\mathbf{u} \in \bV$,
\begin{equation}
\label{eq:c_0_h}
\|\bu -\pi_h \bu\| \le C_{0,h} \|\nabla \bu \|  \quad  \quad (C_{0,h}=O(h)) \:.
\end{equation}
It is easy to see that the constant $\widehat{C}_{0,h}$ in the following inequality provides an upper bound for $C_{0,h}$.
\begin{equation}
\label{eq:c_0_h_upper}
\|u -\pi_h u\| \le \widehat{C}_{0,h} \|\nabla u \|  \quad  \forall u \in H^1(\Omega) \:.
\end{equation}
Here, by using the same notation as in (\ref{eq:c_0_h}), $\pi_h$ denotes the projection $\pi_h: L^2(\Omega) \to X_h$. The explicit bounds of $\widehat{C}_{0,h}$ and $C_{0,h}$ are given in \secref{sec:numerical-results}.

\medskip

Following the idea of \cite{Kikuchi+Saito2007}, 
let us start with $\bff_h:=\pi_h \bff \in \bX_h$ 
with an auxillary boundary value problem: 
Find $\overline{\bu} \in \bV$ such that
\begin{equation}
  \label{eq:auxillary_problem}
(\nabla \overline{\bu} , \nabla \bv ) = (\bff_h, \bv) \quad \forall \bv \in \bV\:.
\end{equation}
The estimate of $\|\nabla (\bu-\overline{\bu})\|$ can be obtained by applying 
the estimation (\ref{eq:c_0_h}) of $\pi_h$  to  $(\ref{eq:stokes}) - (\ref{eq:auxillary_problem})$:
$$
(\nabla (\bu-\overline{\bu}), \nabla \bv) = (\bff - \bff_h, \bv)
= (\bff - \bff_h, (I-\pi_h)\bv)
\le C_{0,h} \|\bff - \bff_h\| \|\nabla \bv \| \:.
$$
By further taking $\bv :=\bu-\overline{\bu}$, we have
\begin{equation}
\label{eq:local-3}
\|\nabla (\bu-\overline{\bu})\| \le C_{0,h} \|\bff-\bff_h\| \:.
\end{equation}

Due to the properties in (\ref{eq:fem_space_properties}), 
for any $\phi_h \in U_h$ and $\bff_h \in X_h$, 
we can find $\bp_h \in \bRT^h$ such that the following equation holds.
\begin{equation}
  \label{eq:ph_cond}
  \Divv \mathbf{p}_h + \nabla \phi_h +  \mathbf{f}_h=0 \:.
\end{equation}
%



\medskip

\subsection{{\em A posteriori} error estimation}

Let us consider the {\em a posteriori} error estimation based on the hypercircle in (\ref{eq:hypercircle-eq}).

\begin{theorem}[A posteriori error estimation]
  \label{thm:a-posteriori-est}
For $\bff \in L^2(\Omega)^3$, let $\bu$ be the 
exact solution to the Stokes equation corresponding to $\bff$.
Let $p_h \in \bRT_h$ be an approximation to $\nabla \bu$ satisfying the condition in (\ref{eq:ph_cond}).
Then we have an  {\em a posteriori} error estimation for both $\bu_h$ and $\bp_h$
  $$
  \|\nabla (\bu-\bu_h)\|, ~ \|\nabla \bu -\bp_h\| \le \|\bp_h - \nabla \bu_h\| + C_{0,h} \|\bff - \bff_h\| \:.
  $$
\end{theorem}
 
\begin{proof}
Let $\overline{\bu}$ be the solution of (\ref{eq:auxillary_problem}).
Replace $\bff$ with $\bff_h$ in  Lemma \ref{lemma:prager-synge}, then we have
\begin{equation}
\label{eq:local-1}  
\|\nabla (\overline{\bu}-{\bu}_h) \| \le \|\bp_h- \nabla {\bu}_h \|, \quad 
\|\nabla \overline{\bu} - \bp_h \| \le \|\bp_h- \nabla {\bu}_h \|\:.
\end{equation}
By applying the triangle inequality, we have 
\begin{equation}
  \label{eq:local-2}
\|\nabla (\bu-\bu_h)\| \le 
\|\nabla (\bu-\overline{\bu}) \| +\|\nabla (\overline{\bu}- {\bu}_h) \|\:.
\end{equation}
With the estimation in (\ref{eq:local-3}) and the first inequality of (\ref{eq:local-1}), we have,
$$
\|\nabla (\bu-\bu_h)\| \le \|\bp_h - \nabla \bu_h\| + C_{0,h} \|\bff - \bff_h\| \:.
$$
Similarly, the estimation for $\|\nabla \bu -\bp_h\|$ is obtained by noticing 
$$
\|\nabla \bu - \bp_h \| \le 
\|\nabla (\bu-\overline{\bu}) \| + \|\nabla \overline{\bu}  - \bp_h \|\:.
$$
\end{proof}

\subsection{{\em A priori} error estimation}

Let us introduce a quantity $\kappa_h$, which will play an important role in the {\em a priori} error estimation.
\begin{equation}
  \label{eq:def-kappa_h}
\kappa_h = \max_{\bff_h \in \mathbf{X}_h} ~~
\min_{\bp_h \in \bRT_h,  \mathbf{v}_h \in \bV_{h}}
\frac{\|\mathbf{p}_h - \nabla \mathbf{v}_h\|}{\|\bff_h\|}\:,
\end{equation}
where the minimization with respect to $\mathbf{p}_h$ is subject to the condition (\ref{eq:ph_cond}).

\medskip

By utilizing the quantity $\kappa_h$, we obtain the {\em a priori} error estimation for FEM solution.
\begin{theorem}
\label{eq:main-theorem-a-prirori-est}
Given $\mathbf{f} \in L^2(\Omega)^3$, let $\bu$ be the exact solution to the Stokes equation and $\bu_h = P_h \bu$. Then, we have 
\begin{equation}
\label{eq:c_h}
\| \B{u} - \B{u}_h \| \le C_h \|\nabla \B{u} - \nabla \B{u}_h \|,
\quad
\|\nabla \B{u} - \nabla \B{u}_h \| \le C_h \|\B{f}\| 
 \:.
\end{equation}
Here, 
$$
C_h :=\sqrt{C_{0,h}^2 + \kappa_h^2}\:.
$$
\end{theorem}
\begin{proof}
  Take $\bff_h:=\pi_h \bff \in \bX_h$ and let $\overline{\bu}$ be the 
exact solution to the Stokes equation corresponding to $\bff_h$ and $\overline{\bu}_h = P_h \overline{\bu}$.
From the definition of projection $P_h$, the hypercircle  (\ref{eq:hypercircle-eq}) and the definition of $\kappa_h$, we have
\begin{equation}
\label{eq:local-a-1}
\|\nabla (\overline{\bu} - \overline{\bu}_h ) \| = 
\min_{\bv_h \in \bV_h} \|\nabla (\overline{\bu} - {\bv}_h ) \|
 \le \min_{\bv_h \in \bV_h} \min_{\bp_h \in \bRT_h } \|\nabla {\bv}_h  - \bp_h \|  \le \kappa_h \|\bff_h\|
\end{equation}
where the minimization w.r.t. $\bp_h$ is subject to the condition (\ref{eq:ph_cond}).\\


By applying the minimization principle to $\bu_h = P_h \bu$ and the triangle inequality,  we have
$$
\|\nabla (\bu-\bu_h)\| \le 
\|\nabla (\bu-\overline{\bu}_h) \| \le 
\|\nabla (\bu-\overline{\bu}) \| +\|\nabla (\overline{\bu}-\overline{\bu}_h) \|\:.
$$
Then, we can draw the conclusion from the estimation in (\ref{eq:local-3}) and (\ref{eq:local-a-1}).
$$
\|\nabla (\bu-\bu_h)\| 
\le  C_{0,h} \|\bff - \bff_h\|  + \kappa_h \|\bff_h\| 
\le \sqrt{C_{0,h}^2 + \kappa_h ^2} ~ \|\bff\|.
$$
The estimation of $\| \B{u} - \B{u}_h \|$ can be obtained by applying the standard the Aubin-Nitsche duality method.
\end{proof}

\subsection{Computation of $\kappa_h$}
\label{sec:kappa-h-computation}

In this subsection, we explain how to calculate the quantity $\kappa_h$.
Given $\bff_h \in \bX_h$, let $\overline{\bu}$ be the 
exact solution to the Stokes equation corresponding to $\bff_h$, i.e.,
$$
(\nabla \overline{\bu}, \nabla \bv) = (\bff_h, \bv), \quad \forall \in \bV. 
$$
Then, from the hypercircle 
$$
\|\nabla \overline{\bu} -\nabla \bv_h\|^2+
\|\nabla \overline{\bu} -\bp_h\|^2
=\|\nabla \bv_h - \bp_h\|^2\:,
$$
we know that to minimize $\|\nabla \bv_h - \bp_h\|$ is equivalent to solve the following two problems.
$$
\min_{\bv_h \in \bV_h} \|\nabla \overline{\bu} -\nabla \bv_h\|,\quad
\min_{\substack{ \bp_h \in \bRT_h \mbox{\footnotesize ~subject to (\ref{eq:ph_cond})}} } \| \nabla \overline{\bu} - \bp_h\|\:.
$$
The minimizer $\bu_h$ and $\bp_h$ can be obtained by solving the following weak problems.

\begin{itemize}
\item [1)] Find $\bu_h \in \bV_h$ s.t.
$$
(\nabla {\bu}_h, \nabla \bv_h) = (\bff_h, \bv_h), \quad \forall \bv_h \in \bV_h\:.
$$

Since it is difficult to construct $\bV_h$ explicitly, we solve the problem by using $\bU_h$ and $\bX_h$ \cite{zhang2005new}: 
Find $\bu_h \in \bU_{h,0}$, $\eta_h \in \bX_h$, $c\in R$ s.t.
\begin{equation}
  \label{eq:sub-pro-uh}
  (\nabla {\bu}_h, \nabla \bv_h)
  + ( \mbox{div }\bu_h, \rho_h) 
  + ( \mbox{div }\bv_h, \eta_h) + 
  (c, \rho_h) + (d, \eta_h)
  = (\bff_h, \bv_h)\:,
\end{equation}
for any $\bv_h \in \bU_{h,0}, \rho_h \in \bX_h, d \in R$\:.

 \item [2)] Find $\bp_h \in \bRT_h$, $\eta_h \in \bX_h$, $\phi_h \in U_h$ such that
\begin{equation}
\label{eq:sub-pro-ph}
  (\bp_h, \bq_h) + (\eta_h, \Div \bq_h + \nabla \psi_h) +
  (\Div \bp_h + \nabla \phi_h, \rho_h  ) = (-\bff_h, \rho_h)\:,
 \end{equation}
  for any $\bq_h \in \bRT_h$, $\rho_h \in \bX_h$, and $\psi_h \in U_h$.
 
\end{itemize}

By taking $\bv_h:=\bu_h$ in (\ref{eq:sub-pro-uh}) and $\bq_h:=\bp_h$ in (\ref{eq:sub-pro-ph}), we have the following equalities for the minimizer $\bu_h$ and $\bp_h$.
$$
(\nabla \bu_h, \nabla \bu_h) = (\bff_h, \bu_h), \quad
(\bp_h, \bp_h) = (-\eta_h, \Div \bp_h + \nabla \phi_h) = (\bff_h, \eta_h)\:.
$$
Thus, the error term 
$\|\nabla \bu_h - \bp_h\|$ can be presented by $\bu_h$ and $\eta_h$. 
\begin{eqnarray*}
\|\nabla \bu_h - \bp_h\|^2 &=& (\nabla \bu_h - \bp_h, \nabla \bu_h - \bp_h)  \\
&=&(\nabla \bu_h, \nabla \bu_h) - 2(\nabla \bu_h,\bp_h) + (\bp_h,\bp_h) \\
&=&(\bff_h, \bu_h) + 2( \bu_h, -\bff_h - \nabla \phi_h) + (\bff_h, \eta_h) \\
&=&(\bff_h , \eta_h  - \bu_h)\:.
\end{eqnarray*}

Let $K_1$, $K_2$ and $K_3$ be the linear operators that map
$\bff_h$ to $\bu_h$, $\bp_h$, $\eta_h$, respectively. Then, the quantity $\kappa_h$ can be calculated by solving the
following problems.
$$
	\kappa_h^2 = \max_{\bff_h \in \bX_h} \frac{ \| \nabla (K_1 \bff_h) - K_2 \bff_h\|^2}{\|\bff_h\|^2}
	= \max_{\bff_h \in \bX_h} \frac{ (\bff_h, (K_3 - K_1) \bff_h) }{\|\bff_h\|^2}\:.
$$

For detailed computation of $\kappa_h$, we can refer to 
\cite{Liu+Oishi2013}, where a similar $\kappa_h$ for the Poisson's equation is discussed and the value of 
$\kappa_h$ is solved by solving a matrix eigenvalue problem.

\begin{remark}
  An efficient computation of the approximate value of $\kappa_h$ is possible by applying iteration methods to solve the matrix eigenvalue problem.
  In the iteration process, for an approximate eigenvector $f_h$, one can solve the sub problem 1) and 2) with standard linear solvers for sparse matrices.
  However, the complexity to have a {\em guaranteed estimation} of quantity $\kappa_h$ is much higher than an approximate estimation.   Generally, to give an upper bound of the maximum eigenvalue of eigenvalue problem $Ax=\lambda Bx$, the basic idea is to 
  apply Sylvester's law of inertia to show that $\widehat{\lambda } B - A $ is positive definite for a candidate upper bound $\widehat{\lambda}$.
  In  this process, the explicit forms of $A$ and $B$ are required. 
  Thus, one has to calculate the inverse of sparse matrices to create the matrices corresponding to the linear operator $K_1$ and $K_3$, 
  which is quite time-consuming and requires huge computer memory.
\end{remark}

\section{Application to eigenvalue problem of Stokes operators}

\label{sec:application-to-eigenvalue}

In this section, we apply the {\em a priori} error estimation to the eigenvalue estimation problem for the 
Stokes operator: Find $\bu \in \bV$ and $\lambda \in R$ such that 
\begin{equation}
\label{eq:stokes_eig_pro}
(\nabla \bu, \nabla \bv) = \lambda (\bu, \bv) \quad\forall \bv \in \bV\:.
\end{equation}
Denote the eigenvalues of the above problem by $\lambda_1 \le \lambda_2 \le \cdots$.
The FEM approach for the Stokes eigenvalue problem is as follows.
 Find $\bu_h \in \bV_h$ and $\lambda_h \in R$ such that 
\begin{equation}
\label{eq:stokes_eig_pro_h}
(\nabla \bu_h, \nabla \bv_h) = \lambda_h (\bu_h, \bv_h) \quad\forall \bv_h \in \bV_h\:.
\end{equation}
Denote the approximate eigenvalues by $\lambda_{h,1} \le \lambda_{h,2} \le \cdots \le \lambda_{h,n}$ ($n=\mbox{dim}(\bV_h)$).
Since $\bV_h \subset \bV$, the min-max principle assures that the FEM approximation $\lambda_{h,k}$ gives 
upper bound for the exact eigenvalue $\lambda_k$.

Next theorem is a direct result by applying Theorem 2.1 of Liu \cite{Liu2015} to the Stokes eigenvalue problem (\ref{eq:stokes_eig_pro}).

\begin{theorem}[Application of Theorem 2.1 of \cite{Liu2015}]
  \label{thm:eig-lower-bound}
Let $P_h$ be the project defined in (\ref{eq:def-Ph}). Then, the eigenvalue $\lambda_{k}$ of (\ref{eq:stokes_eig_pro}) has a lower bound as 

\begin{equation}
\label{eq:eig_lower_bound}
\lambda_k \ge \frac{\lambda_{h,k}}{1+C_h^2 \lambda_{h,k}} (=: \underline{\lambda}_{h,k})
,\quad k=1,2,\cdots, \mbox{dim}(\bV_h).
\end{equation}

\begin{remark}
The original Theorem 2.1 in \cite{Liu2015} also works for non-conforming FEM space. In \cite{Xie2-Liu-JJIAM-2018}, the Crouzeix-Raviart 
non-conforming FEM  is utilized to bound the eigenvalue of the Stokes operator on 2D domains.
In case of 3D domains, let the approximate eigenvalue obtained by the Crouzeix-Raviart FEM space be $\lambda_{h,k}^{NC}$, then we have
\begin{equation}
\label{eq:eig_lower_bound}
\lambda_k \ge \frac{\lambda_{h,k}^{NC}}{1+(0.3804h)^2 \lambda_{h,k}^{NC}} (=: \underline{\lambda}_{h,k}^{NC})
,\quad k=1,2,\cdots, \mbox{dim}(\bV_h).
\end{equation}
Here, the quantity $0.3804$ comes from the estimation in \cite{Liu2015}. 
The lower bounds based on the {\em a priori} estimation and the one based on the non-conforming FEM space 
are compared in the section of numerical computations.
\end{remark}

\end{theorem}

\section{Numerical computations}
\label{sec:numerical-results}

In this section, we solve the Stokes equation over several 3D domains, including the cube domain, the L-shape domain and the cube-minus-cube domain. 

To have a stable computing of the Stokes equation, we apply Zhang's method \cite{zhang2005new} in the mesh generation process.
First, each domain is subdivided into uniform small cubes. Then each cube is divided into 5 tetrahedra. Finally, by following Zhang's method, each tetrahedron is partitioned into $4$ sub-tetrahedra  by using the barycentric of the tetrahedron. Note that for 2D case, a mesh without degenerate point is required for a stable computation \cite{scott1985norm}.

Let $h$ be the edge of length of small cubes in the subdivision. 
From the results in \cite{JCAM-LIU-2020}, we have an upper bound for the Poincar\'e constant over the special tetrahedra 
resulted by Zhang's subdivision method. 
\begin{equation}
  \label{eq:poincare-constant-bound}
  C_{0,h} \le \widehat{C}_{0,h} \le 0.284 h ~~ (h: \mbox{the largest edge length of sub-cubes}) .
\end{equation}

\medskip

For the FEM spaces over the mesh, the degrees of FEM function spaces are selected such that $d=m=k-1(\ge 2)$. The selection of $k \ge 3$ will make sure the regularity of 
matrices in the computation. However, since many linear solvers can also deal with the cases with singular matrices, 
we also show the computation results for $d=m=1, k=2$. Notice that it is difficult to obtain guaranteed estimation for $\kappa_h$ with $k=2$.

To have rigorous bounds for the estimation of $\kappa_h$ and $C_h$, we apply the verified computation method in the computation and use the INTLAB toolbox \cite{INTLAB} for interval arithmetic. 
However, since the algorithm in verified computation of eigenvalues involves inverse computation of sparse matrices, the verified computation requires huge computer memory and the computation takes longer time than approximate computation. See the comparison of resource consuming in Table \ref{table:approximate_vs_verified_computation}. 
For the approximate estimation of $\kappa_h$, the involved matrix eigenvalue problem for a dense mesh is solved by using the SLEPc eigenvalue solver in the PETSc library \cite{slepc-toms,petsc-efficient}.

In the last subsection, we also solve the eigenvalue problem (\ref{eq:stokes_eig_pro}) on a cube domain to estimate the Poincar\'e constant over the divergence-free space $V$. 

\subsection{{\em A priori} error estimation over 3D domains}

This subsection displays the {\em a priori} error estimation results for several 3D domains. 
In the tables of the computation results, \#Elt denotes the number of elements and 
DOF denotes the largest degree of freedoms (DOF) of the FEM spaces involved in evaluating the quantity $\kappa_h$.

\paragraph{Cube domain} For the unit cube domain $\Omega=(0,1)^3$, the computation results with $k=2$ and $k=3$ are listed in Table \ref{table:cube_domain_results_k_as_2} and 
\ref{table:cube_domain_results_k_as_3}, respectively. 
In this case, we confirm the convergence rate of $C_h$ in both cases is about one.


\begin{table}[h]
\begin{center}
\caption{\label{table:cube_domain_results_k_as_2} The {\em a priori} error estimation (Cube domain, $d=m=1,k=2$) }
\begin{tabular}{c|c|c|ccc|c}
\hline
\rule[-2mm]{0cm}{0.6cm}
N & \#Elt & DOF    & $\kappa_h$ & $C_{0,h}$ & $C_h$ & Order \\
\hline
1 & 20    & 886    & 1.34E-1 & 2.84E-1 & 3.14E-1 & - \\
2 & 160   & 6774   & 1.04E-1 & 1.42E-1 & 1.76E-1 & 0.84   \\
4 & 1280  & 53050  & 5.54E-2 & 7.10E-2 & 9.01E-2 & 0.97 \\
8 & 10240 & 420018 & 2.83E-2 & 3.55E-2 & 4.54E-2 & 0.99 \\
\hline
\end{tabular}
\end{center}
\end{table}

\begin{table}[h]
\begin{center}
\caption{\label{table:cube_domain_results_k_as_3} The {\em a priori} error estimation (Cube domain, $d=m=2,k=3$) }
\begin{tabular}{c|c|ccc|c}
\hline
\rule[-2mm]{0cm}{0.6cm}
N & DOF & $\kappa_h$ & $C_{0,h}$ & $C_h$ & Order  \\
\hline
1 & 2284 & \underline{1.22E-1}  & \underline{2.84E-1} & \underline{3.09E-1} & - \\ 
2 & 17664    & \underline{7.99E-2} & \underline{1.42E-1} & \underline{1.63E-1} & 0.99  \\
3 & 139030   & 3.23E-2 & 7.10E-2 & 7.80E-2 & 1.00 \\
4 & 1103370  & 1.61E-2 & 3.55E-2 & 3.90E-2 & 1.00 \\
\hline
\end{tabular}\\
\rule[-3mm]{0cm}{0.7cm}{\footnotesize
Note: Underlined numbers are rigorous bounds by using verified computation. Approximate esitmation shows that $\kappa_h\approx 0.121$ for $N=1$, $\kappa_h\approx 0.0652$ for $N=2$.}
\end{center}
\end{table}

\begin{table}[h]
  \begin{center}
    \caption{\label{table:approximate_vs_verified_computation}Resource consuming comparison between approximate scheme and verified computation}    
  \begin{tabular}{|c|c|c|}
    \hline
    \rule[-2mm]{0cm}{0.6cm}
    & Approximate scheme & Verified computation \\
    \hline
    \rule[-2mm]{0cm}{0.6cm}
    Computer memory & 19.7MB  & 205.8MB \\
    \hline
    \rule[-2mm]{0cm}{0.6cm}
    Computing time  & 0.5 second & 54.8 seconds \\
    \hline
    \end{tabular}
\end{center}
\end{table}

\paragraph{L-shaped domain} We consider the L-shaped domain $\Omega:=((-1,1)^2 \setminus [-1,0]^2)\times (0,1)$. 
The domain is consisted of three unit cubes. 
The {\em a priori} error estimation results are displayed in Table \ref{table:L_domain_results_k_as_2} and \ref{table:L_domain_results_k_as_3}.
Here, $N$ denotes the number of sub-cubes along $z$ direction in the subdivision process.
Since the domain is non-convex, the solution of Stokes equation may have singularity around the re-entrant boundary, which will cause a dropped convergence rate.

\begin{table}[h]
\begin{center}
\caption{\label{table:L_domain_results_k_as_2} The {\em a priori} error estimation and eigenvalue error bounds (L-shaped domain, $d=m=1,k=2$) }
\begin{tabular}{c|cc|cccc}
\hline
\rule[-2mm]{0cm}{0.6cm}
N & \#Elt & DOF    & $\kappa_h$ & $C_{0,h}$ & $C_h$ & Order \\
\hline
1 & 60    & 2640    & 1.36E-1 & 2.84E-1 & 3.15E-1 & - \\
2 & 480   & 20126   & 1.27E-1 & 1.42E-1 & 1.91E-1 & 0.7\\
3 & 3840  & 158410  & 7.67E-2 & 7.10E-2  & 1.05E-1  & 0.9\\
4 & 30720 & 1257170 & 4.83E-2 & 3.55E-2 & 5.99E-2 & 0.8\\
\hline
\end{tabular}
\end{center}
\end{table}
\vskip 1cm

\begin{table}[h]
\begin{center}
\caption{\label{table:L_domain_results_k_as_3} The {\em a priori} error estimation and eigenvalue error bounds (L-shaped domain, $d=m=2,k=3$) }
\begin{tabular}{c|c|ccc|c}
\hline
\rule[-2mm]{0cm}{0.6cm}
N & DOF & $\kappa_h$ & $C_{0,h}$ & $C_h$ & Order \\
\hline
1 & 6746 & \underline{1.38E-1} & \underline{2.84E-1} & \underline{3.16E-1} & - \\ 
2 & 52604 & 8.00E-2 & 1.42E-1  & 1.63E-1 & 0.96 \\
4 & 415598 & 4.43E-2 & 7.10E-2 & 8.37E-1 & 0.96 \\
8 & 3304250 & 2.92E-2 & 3.55E-2 & 4.59E-2 & 0.87 \\ 
\hline
\end{tabular}\\
\rule[-3mm]{0cm}{0.7cm}{\footnotesize
Note: Approximate esitmation shows that $\kappa_h\approx$ 1.22E-1 for $N=1$.}
\end{center}
\end{table}

\paragraph{Cube-minus-cube domain} 
Let us consider the Stokes equation over a cube-minus-cube domain $\Omega:=((0,2)^3 \setminus [1,2]^3)$, which 
is consisted of $7$ unit cubes. 
In the mesh generation process, denote by $N$ the number of sub-cubes along $z$ direction. 
Notice that the initial domain has $N=2$ while the edge length of sub-cube is $1$. 
The computational results are listed in Table \ref{table:cube_minus_cube_domain_results_k_2} and \ref{table:cube_minus_cube_domain_results_k_3}.
To have verified bound for $\kappa_h$, it takes about 2.4 hours and about 9GB memory for the case $N=2$.
Moreover, due to the accumulation of rounding error in verified computing, 
the estimated rigorous bound of $\kappa_h$ is about $1.4$ times of the approximation obtained by the classical numerical library SLEPc \cite{slepc-toms}.

\begin{table}[h]
  \begin{center}
  \caption{\label{table:cube_minus_cube_domain_results_k_2} Quantities in the {\em a priori} error estimation (Cube-minus-cube  domain, $d=m=1,k=2$) }
  \begin{tabular}{c|cc|ccc|c}
  \hline
  \rule[-2mm]{0cm}{0.6cm}
  N & \#Elt & DOF      & $\kappa_h$ & $C_{0,h}$ & $C_h$ & Order \\
  \hline
  2 & 140   &  5962   & 2.02E-1 & 2.84E-1  & 3.49E-1 & - \\ 
  4 & 1120  &  46554  & 1.33E-1   & 1.42E-1 & 1.94E-1 & 0.85 \\
  8 & 8960  &  368050  & 8.19E-2   & 7.10E-2 &  1.08E-1  & 0.84 \\
  16 & 8960 &  2927202 & 5.30E-2   & 3.55E-2 & 6.38E-2 & 0.77 \\
  \hline
  \end{tabular}\\
  \rule[-3mm]{0cm}{0.7cm}
  \end{center}
  \end{table}
  \vskip 1cm
  
\begin{table}[h]
  \begin{center}
  \caption{\label{table:cube_minus_cube_domain_results_k_3} Quantities in the {\em a priori} error estimation (Cube-minus-cube  domain, $d=m=2,k=3$) }
  \begin{tabular}{c|c|ccc|c}
  \hline
  \rule[-2mm]{0cm}{0.6cm}
  N  & DOF      & $\kappa_h$ & $C_{0,h}$ & $C_h$ & Order \\
  \hline
  2  &  15526   & \underline{1.85E-1} & \underline{2.84E-1}  & \underline{3.39E-1} & - \\ 
  4  &  121926  & 8.52E-1   & 1.42E-1 & 1.66E-1 & 1.03 \\
  8  &  966538  & 4.64E-1   & 7.10E-2 &  8.48E-2  & 0.97 \\
  16 &  7697298 & 3.19E-2   & 3.55E-2 & 4.77E-2 & 0.83 \\
  \hline
  \end{tabular}\\
  \rule[-3mm]{0cm}{0.7cm}
  {\footnotesize
  Note: Approximate estimation shows that $\kappa_h\approx$ 1.57E-1 for $N=2$.}
  \end{center}
  \end{table}
  \vskip 1cm

\subsection{Application to the estimation of Poincar\'e constant}
\label{sub-sec:application-to-eigenvalue-pro}

We apply the {\em a priori } to estimate the Poincar\'e constant by solving the corresponding eigenvalue problem of the Stokes operator over the unit cube domain. 
The Poincar\'e constant $C_p$ over $\bV$ is the optimal quantity to make the following inequality hold. 
\begin{equation}
\|\bv \|_{L^2(\Omega)}   \le C_p \|\nabla \bv \|_{L^2(\Omega)}\quad \forall \bv \in \bV\:.
\end{equation}
Th constant $C_p$ is determined by the first eigenvalue of the Stokes operator defined in (\ref{eq:stokes_eig_pro}). That is, 
$C_p=1/\sqrt{\lambda_1}$

Recall the Poincar\'e constant $C_p'$ over $H_0^1(\Omega)$ such that,
\begin{equation}
  \label{eq:cp_dash}
  \|v \| \le C_p' \| \nabla v \| \quad \forall v \in H_0^1(\Omega)\:.  
\end{equation}
It is easy to see that $C_p \le C_p'$. 
The estimation of $C_p'$ has the optimal value as $C_p'=\sqrt{3}/(3\pi)\approx 0.183776\cdots$ for the unit cube domain. The equality in (\ref{eq:cp_dash}) holds for $v=\sin\pi x \sin\pi y \sin\pi z$.

By applying the lower bound estimation in (\ref{eq:eig_lower_bound}) along with the explicit value of projection error quantity $C_h$ in Table \ref{table:cube_domain_results_k_as_3},
we obtain the two-side bounds for both $\lambda_1$ and $C_p$.
As a comparison, the lower bounds based on the Crouzeix-Raviart nonconforming method are also provided.
The estimation results are displayed in Table \ref{table:poincare_constant_3d}. 

\begin{table}[h]
\begin{center}
\caption{\label{table:poincare_constant_3d}  
Estimation of eigenvalue and the Poincar\'e constant over the unit cube domain}
\begin{tabular}{c|ccc|c|c}
\hline
\rule[-2mm]{0cm}{0.6cm}
h &  $C_h$ & $\lambda_{h,1}$ & $\underline{\lambda}_{h,1}$ & $\underline{\lambda}^{NC}_{h,1}$ & Estimation of $C_p$ \\
\hline
\rule[-1mm]{0cm}{0.6cm}
1   & 3.14E-1 & 189.46 & 9.62 & 2.89 & 0.072 $\le C_p \le $ 0.323  \\
\rule[-1mm]{0cm}{0.6cm}
1/2 & 1.76E-1 & 65.06 & 21.63 & 9.82 & 0.123 $\le C_p \le $ 0.215\\
\rule[-1mm]{0cm}{0.6cm}
1/4 & 9.01E-2 & 62.31 & 41.37 & 27.05& 0.126 $\le C_p \le $ 0.156 \\
\rule[-1mm]{0cm}{0.6cm}
1/8 & 4.54E-2 & 62.18 & 55.11 & 46.97 & 0.127 $\le C_p \le $ 0.135\\
 \hline
\end{tabular}\\
\rule[-3mm]{0cm}{0.7cm}
{\footnotesize
Note: $C'_p \approx 0.183776$ gives a rough bound for $C_p$.}
\end{center}
\end{table}

\section*{Conclusion}
For the Stokes equation over general 3D domains, an explicit error estimation for the finite element approximation is proposed. 
The difficulty in dealing with the divergence-free condition is solved by introducing the extended hypercircle method and utilizing the 
conforming FEM space provided by the Scott-Vogelius scheme and Zhang's mesh. 
The resulted estimation plays an important role in the computer-assisted proof for the solution existence verification to the Navier--Stokes equation, which will be challenged in the next step of our research.
\vskip 0.5cm
  
\bibliographystyle{plain}
\bibliography{mylib}

\end{document}